\documentclass[12pt]{amsart}
\usepackage{mathrsfs}
\usepackage{amsmath,amssymb,amsfonts,latexsym,txfonts}
\usepackage{eucal}
\usepackage{fancyhdr}

\newlength{\originalbase}
\newcommand{\spacing}[1]{\setlength{\baselineskip}{#1\originalbase}}

\newtheorem{theorem}{Theorem}[section]
\newtheorem{proposition}[theorem]{Proposition}
\newtheorem{lemma}[theorem]{Lemma}
\newtheorem{definition}[theorem]{Definition}
\newtheorem{corollary}[theorem]{Corollary}

\input amssym.def
 
\input amssym.tex

\renewcommand{\SS}{\mathbb{S}}
\newcommand{\RR}{\mathbb{R}}

\newcommand{\LL}{\mathcal{L}}
\newcommand{\KK}{\mathscr{K}}

\newcommand{\bs}{{\,\#\,}}
\newcommand{\cy}{\mathcal{C}}

\begin{document}

\begin{center}
{\bf\Large  Covering shadows with a smaller volume}\\[12mm]
\end{center}

\begin{center}
{\bf Daniel A. Klain} \\
Department of Mathematical Sciences\\
University of Massachusetts Lowell\\
Lowell, MA 01854 USA  \\
Daniel\_{}Klain@uml.edu\\
\end{center}

\vspace{4mm}
\begin{center}
{\small\em Larger things can hide behind smaller things.}
\end{center}
\vspace{6mm}
\spacing{1}

\begin{quotation}
{\small
\noindent
{\bf Abstract \/}  
For $n \geq 2$ a construction is given for 
convex bodies $K$ and $L$ in $\RR^n$ such that the orthogonal projection 
$K_u$ can be translated inside $L_u$ 
for every direction $u$, while the volumes of $K$ and $L$ satisfy $V_n(K) > V_n(L).$

A more general construction is then given for 
$n$-dimensional convex bodies $K$ and $L$ such that the orthogonal projection 
$K_\xi$ can be translated inside $L_\xi$ 
for every $k$-dimensional subspace $\xi$ of $\RR^n$, while 
the $m$-th intrinsic volumes of $K$ and $L$ satisfy
$V_m(K) > V_m(L)$ for all $m > k$.

It is then shown that, for each $k = 1, \ldots, n$, there is a class of bodies $\cy_{n,k}$ 
such that, if $L \in \cy_{n,k}$ and if the orthogonal projection 
$K_\xi$ can be translated into $L_\xi$ 
for every $k$-dimensional subspace $\xi$ of $\RR^n$, then $V_n(K) \leq V_n(L)$.

The families $\cy_{n,k}$, called $k$-cylinder bodies of $\RR^n$, form a strictly
increasing chain 
$$\cy_{n,1} \subset \cy_{n,2} \subset \cdots \subset \cy_{n,n-1} \subset \cy_{n,n},$$
where $\cy_{n,1}$ is precisely the collection of centrally symmetric compact
convex sets in $\RR^n$, while $\cy_{n,n}$ is the collection of all compact
convex sets in $\RR^n$.  Members of each family $\cy_{n,k}$ are seen to play
a fundamental role in relating covering conditions for projections to the theory of mixed volumes,
and members of $\cy_{n,k}$ are shown to satisfy certain geometric inequalities.
Related open questions are also posed.
}
\end{quotation}
\vspace{6mm}

Suppose that $K$ and $L$ are compact convex subsets of 
$n$-dimensional Euclidean space.
For a given dimension $1 \leq k < n$,
suppose that every $k$-dimensional orthogonal projection (shadow) of $K$
can be translated inside the corresponding projection of $L$.
Does it follow that $K$ has smaller volume than $L$?  In this article it is shown
that the answer in general is {\em no}.  It is then shown that the answer is {\em yes} if $L$ is chosen from
a suitable family of convex bodies that includes certain cylinders and other sets with a direct sum
decomposition.

Many inverse questions from convex and integral geometry take the following form:  Given two convex bodies $K$ and $L$,
and two geometric invariants $f$ and $g$ (such as volume, or surface area, or some measure of sections or projections),
does $f(K) \leq f(L)$ imply $g(K) \leq g(L)$?  If not, then what additional conditions on $K$ and $L$ are necessary?  

These questions are motivated in part
by the projection theorems of Rogers \cite{rogers}.
Rogers showed that if two compact convex sets
have translation congruent (or, more generally, homothetic) projections in every linear subspace of some chosen
dimension $k \geq 2$, then the original sets $K$ and $L$ must be translation congruent (or homothetic).
Rogers also proved analogous results for sections of sets with hyperplanes through a base point \cite{rogers}.
These results then set the stage for more general (and often much more difficult) questions, in which 
the rigid conditions of translation congruence or homothety are 
replaced with weaker conditions, such as containment up to translation, inequalities of measure, etc.

Two notorious questions of this kind are the Shephard Problem \cite{Shep}
(solved independently by  
Petty \cite{Petty-shep} and Schneider \cite{Schneider-shep}), 
and the Busemann-Petty Problem \cite{bus-pet} 
(solved in work of Gardner \cite{Gard-Busemann}, 
Gardner, Koldobsky, and Schlumprecht \cite{Gard-Kold}, and 
Zhang \cite{Zhang-Busemann-old,Zhang-Busemann}).  
Both questions address properties of bodies $K$ and $L$ 
that are assumed to be {\em centrally symmetric} about the origin.  

The Shephard Problem asks: if 
the $(n-1)$-dimensional volumes of the orthogonal projections 
$K_u$ and $L_u$ of convex bodies $K$ and $L$ satisfy the inequality 
$V_{n-1}(K_u) \leq V_{n-1}(L_u)$ for every direction $u$, does it follow that $V_n(K) \leq V_n(L)$?
Although there are ready counter-examples for general (possibly non-symmetric) convex bodies, 
the problem is more difficult to address
under the stated assumption that $K$ and $L$ are both centrally symmetric.  In this case  
Petty and Schneider have shown that, while the answer in general is still no for dimensions $n \geq 3$, 
the answer is yes when the convex set $L$ is a projection body; that is, a zonoid.  

The Busemann-Petty Problem addresses the analogous
question for sections through the origin.  
Suppose that convex bodies $K$ and $L$ are centrally symmetric about the origin.
If we assume that
the $(n-1)$-dimensional sections of $K$ and $L$ satisfy
$$V_{n-1}(K \cap u^{\perp}) \leq V_{n-1}(L \cap u^{\perp})$$ 
for every direction $u$, does it follow that $V_n(K) \leq V_n(L)$?  Surprisingly the answer is 
{\em no} for bodies of dimension $n \geq 5$ and {\em yes}
for bodies of dimension $n \leq 4$ (see
\cite{Gard-Busemann,Gard-Kold,Zhang-Busemann-old,Zhang-Busemann}).  
Moreover, Lutwak \cite{Lut1988} has shown that,
in analogy to the Petty-Schneider theorem, 
the answer is always {\em yes} when the set $L$ is an {\em intersection body}, a construct 
highly analogous to projection bodies (zonoids), but for which projection (the cosine transform) 
is replaced in the construction with intersection (the Radon transform).
A more complete discussion of background to 
the Busemann-Petty Problem, its solution, and its variations (some of which remain open), can be
found in the comprehensive book by Gardner \cite{Gard2006}.

Both of the previous problems assume that bodies in question are either
centrally symmetric or symmetric about the origin; 
that is, $K = -K$ and $L = -L$ (up to translation).  If this elementary assumption is omitted, 
then both questions are easily seen to have negative answers.  
For the projection problem, compare the Reuleaux triangle, 
and its higher dimensional analogues, with the Euclidean ball, or compare 
any non-centered convex body with its Blaschke body \cite{Gard2006}.
For the intersection problem, consider a non-centered planar set having an equichordal point,
or the dual analogue of the Blaschke body of a non-centered  set (See \cite[p. 117]{Gard2006}
or \cite{Lut1988}).  

In the present article we consider a related, but fundamentally different, family of questions.

Suppose that, instead of comparing the areas of the projections of $K$ and $L$, 
we assume that the projections of $L$ {\em can cover} translates of the projections of $K$. 
Specifically, suppose that, 
for each direction $u$,
the orthogonal projection
$K_u$ of $K$ can be translated so that it is 
contained inside the corresponding projection $L_u$ 
(although the required translation may vary depending on $u$).  
Does it follow that $K$ can be translated so that it is contained inside $L$?  
Does it even follow that $V_n(K) \leq V_n(L)$?
  
These questions have easily described negative answers in dimension 2, 
since the projections are 1-dimensional, and convex 1-dimensional sets have very little structure.  
(Once again, consider the Reuleaux triangle and the circle.)  
The interesting cases begin when comparing 2-dimensional projections of 3-dimensional objects, 
and continue from there.

For higher dimensions, a simple example illustrates once again 
that $K$ might {\em not} fit inside $L$, 
even though every projection of $L$ can be translated to 
cover the corresponding projection of $K$.  
Let $L$ denote the unit Euclidean 3-ball, 
and let $K$ denote the regular tetrahedron having edge length $\sqrt{3}$.  
Jung's Theorem \cite[p. 84]{Bonn2}\cite[p. 320]{Webster}
implies that every 2-projection of $K$ 
is covered by a translate of the unit disk.  But a simple computation shows that 
$L$ does not contain a translate of tetrahedron $K$.  
An analogous construction yields a similar result for higher dimensional simplices and Euclidean balls. 
One might say that, although $K$ can ``hide behind" $L$ from every observer's perspective, 
this does not imply that $K$ can hide {\em inside} $L$. 

In the previous counterexample it is still the case that the set $L$ having larger (covering) shadows 
also has {\em larger volume} than $K$.  Although the question of comparing volumes
is more subtle, there are counterexamples to this property as well.  

This article presents the following results for every dimension $n \geq 2$:
\begin{enumerate}
\item[$\mathbf{1}$.] There exist $n$-dimensional convex bodies $K$ and $L$ such that the orthogonal projection 
$K_u$ can be translated inside $L_u$ 
for every direction $u$, while $V_n(K) > V_n(L).$\\

\item[$\mathbf{2}$.] There is a large class of bodies $\cy_{n,n-1}$ 
such that, if $L \in \cy_{n,n-1}$ and if
$K_u$ can be translated inside $L_u$ 
for every direction $u$, then $V_n(K) \leq V_n(L)$.
\end{enumerate}
In particular, it will be shown that if the body $L$ having covering shadows is a {\em cylinder},
then $V_n(K) \leq V_n(L)$.  
The more general collection $\cy_{n,n-1}$, called $(n-1)$-cylinder bodies, play a role 
for the covering projection problem in
analogy to that of intersection bodies for the 
Busemann-Petty Problem and that of zonoids for the 
Shephard Problem. 

These results generalize to questions about shadows (projections) of arbitrary lower dimension.  
If $\xi$ is a $k$-dimensional subspace of $\RR^n$,
denote by $K_\xi$ the orthogonal projection of a body $K$ into $\xi$.  For convex bodies $K$ in $\RR^n$
and $0 \leq m \leq n$, denote by $V_m(K)$ the $m$th {\em intrinsic volume} of $K$.
The main theorems of this article also yield the following more general observations,
for each $n \geq 2$ and each $1 \leq k \leq n-1$:
\begin{enumerate}
\item[$\mathbf{1'}$.] There exist $n$-dimensional convex bodies $K$ and $L$ such that the orthogonal projection 
$K_\xi$ can be translated inside $L_\xi$ 
for every $k$-dimensional subspace $\xi$ of $\RR^n$, while $V_m(K) > V_m(L)$ for all $m > k$.\\

\item[$\mathbf{2'}$.] There is a class of bodies $\cy_{n,k}$ 
such that, if $L \in \cy_{n,k}$ and if the orthogonal projection 
$K_\xi$ can be translated inside $L_\xi$ 
for every $k$-dimensional subspace $\xi$ of $\RR^n$, then $V_n(K) \leq V_n(L)$.
\end{enumerate}

\sloppy
The aforementioned counterexamples are constructed in Sections~\ref{counter} and~\ref{gencounter}.
Cylinder bodies and their relation to projection and covering are described in 
Sections~\ref{sec-cover} and~\ref{sec-cyl}, leading to the Shadow Containment Theorem~\ref{omni-k}, 
which relates covering of shadows
to a family of inequalities for mixed volumes.
These developments 
lead in turn to Theorem~\ref{cylvol}, where it is shown that, if every shadow of a cylinder body $L$ 
contains a translate of the corresponding shadow of $K$,
then $L$ must have greater volume than $K$.  In Section~\ref{sec-ineq}
the counterexample constructions of Sections~\ref{counter} and~\ref{gencounter} are used to prove a family 
of geometric inequalities satisfied by members of each collection $\cy_{n,k}$.  Section~\ref{sec-ineq2} uses
Theorem~\ref{cylvol} to prove 
that $V_n(K) \leq nV_n(L)$ whenever the projections of $K$ can be translated inside those of $L$.

The constructions and theorems of this article motivate a number of new open questions related to covering projections,
some of which are posed in the final section.

\section{Preliminary background}
Denote by $\KK_n$ the set of compact convex subsets of $\RR^n$.  The $n$-dimensional
(Euclidean) volume of a convex set $K$ will be denoted $V_n(K)$.  If $u$ is a unit vector in
$\RR^n$, denote by $K_u$ the orthogonal projection of a set $K$ onto the subspace $u^\perp$.

Let $h_K: \RR^n \rightarrow \RR$ denote the support function of a compact convex set $K$;
that is,
$$h_K(v) = \max_{x \in K} x \cdot v$$
If $u$ is a unit vector in
$\RR^n$, denote by $K^u$ the support set of $K$ in the direction of $u$; that is,
$$K^u = \{x \in K \; | \; x \cdot u = h_K(u) \}.$$
If $P$ is a convex polytope, then $P^u$ is the face of $P$ having $u$ in its outer normal cone.

Given two
compact convex sets $K, L \in \KK_n$ and $a,b \geq 0$ denote
$$aK + bL = \{ax + by \; | \; x \in K \hbox{ and } y \in L\}$$
An expression of this form is called a {\em Minkowski combination} or 
{\em Minkowski sum}.  Because $K$ and $L$ are convex, the set $aK + bL$ is also convex.  
Convexity also implies that $aK + bK = (a+b)K$ for all $a,b \geq 0$.

Support functions are easily seen to satisfy the identity
$h_{aK+bL} = ah_K + bh_L$.  Moreover,
the volume of a Minkowski combination of two compact convex sets is given by {\em Steiner's formula:}
\begin{equation}
V_n(aK + bL) = \sum_{i=0}^n \binom{n}{i} \, a^{n-i} b^{i} 
V_{n-i, i}(K, L),
\label{steinform}
\end{equation}
where the {\em mixed volumes} 
$V_{i, n-i}(K, L)$ depend only on $K$ and $L$ and the indices $i$ and $n$.
In particular, if we fix two convex sets $K$ and $L$ then
the function $f(a,b) = V_n(aK + bL)$ is a homogeneous polynomial of degree $n$
in the non-negative variables $a, b$.  

Each mixed volume $V_{n-i, i}(K, L)$ is non-negative, continuous in
the entries $K$ and $L$, and monotonic with respect to set inclusion.
Note also that $V_{n-i, i}(K, K) = V_n(K)$.
If $\psi$ is an affine transformation whose linear component has determinant 
denoted $\det \psi$, then 
$V_{i, n-i}(\psi K, \psi L) = |\det \psi| \, V_{n-i, i}(K, L)$.
If $P$ is a polytope, then the mixed volume $V_{n-1, 1}(P, K)$
satisfies the classical ``base-height" formula 
\begin{equation}
V_{n-1, 1}(P, K) = \frac{1}{n} \sum_{u \perp \partial P} h_K(u) V_{n-1}(P^u),
\label{polyvol}
\end{equation}
where this sum is finite, taken over all outer normals $u$ to the {\em facets} 
on the boundary $\partial P$.
These and many other properties of convex bodies and mixed volumes 
are described in detail in each of \cite{Bonn2,red,Webster}.

The {\em Brunn-Minkowski inequality} asserts that,
for $0 \leq \lambda \leq 1$,
\begin{equation}
V_n((1-\lambda)K + \lambda L)^{1/n} \geq (1 - \lambda)V_n(K)^{1/n} + \lambda V_n(L)^{1/n}.
\label{bmcc}
\end{equation}
If $K$ and $L$ have interior, then equality holds in~(\ref{bmcc}) if and only if $K$ and $L$
are homothetic; that is, 
iff there exist $a \in \RR$ and $x \in \RR^n$ such that $L = aK + x$.   
On combining~(\ref{bmcc}) with Steiner's formula~(\ref{steinform}) one obtains the {\em Minkowski mixed volume inequality: }
\begin{equation}
V_{n-1,1}(K,L)^n \geq V_n(K)^{n-1} V_n(L),
\label{mmv}
\end{equation}
with the same equality conditions as in~(\ref{bmcc}).  See, for example, any of 
\cite{Bonn2,Gard-BM,red,Webster}.

If $K \in \KK_n$ has non-empty interior, define the {\em surface area measure} $S_K$ on
the $(n-1)$-dimensional unit sphere $\SS^{n-1}$ as follows.  For $A \subseteq \SS^{n-1}$
denote by $K^A = \bigcup_{u \in A} K^u$, and define
$S_K(A) = \mathcal{H}_{n-1}(K^A)$, the $(n-1)$-dimensional Hausdorff measure of the
subset $K^A$ of the boundary of $K$.  (See \cite[p. 203]{red}.)

Note that, if $P$ is a polytope, then $S_P$ is a pointed measure concentrated 
at precisely those directions $u$ that are outer normals to the facets of $P$.

The measure $S_K$ is easily shown to satisfy the property that
\begin{equation}
\int_{\SS^{n-1}} u \; dS_K = \vec{o},
\label{surf}
\end{equation}
that is, the mass distribution on the sphere described by $S_{K}$ has center of mass
at the origin.  The identity~(\ref{polyvol}) can now be expressed in its more general form:
\begin{equation}
V_{n-1, 1}(K, L) = \frac{1}{n} \int_{\SS^{n-1}} h_L(u) \; dS_K(u),
\label{anyvol}
\end{equation}
for all convex bodies $K$ and $L$ such that $K$ has non-empty interior.  It follows 
from~(\ref{anyvol}) and the Minkowski linearity of the support function that, for $K,L,M \in \KK_n$ and $a,b \geq 0$,
\begin{equation}
V_{n-1, 1}(K, aL+bM) = aV_{n-1, 1}(K, L) + b V_{n-1, 1}(K, M).
\label{left}
\end{equation}
If $B$ is a unit Euclidean ball centered at the origin, then $h_B = 1$ in every direction, so that $n V_{n-1, 1}(K, B) =  S(K)$, the surface area of the convex body $K$.

Minkowski's Existence Theorem \cite[p. 125]{Bonn2}\cite[p. 390]{red} gives an 
important and useful converse to the identity~(\ref{surf}):  If $\mu$ is a non-negative measure on the unit sphere $\SS^{n-1}$ such that $\mu$ has center of mass at the origin, and if $\mu$ is not concentrated on any great (equatorial) $(n-1)$-subsphere, then $\mu = S_K$ for some $K \in \KK_n$.  Moreover, this convex body $K$ is {\em unique up to translation}.

Minkowski's Existence Theorem provides the framework for the following definition:
For $K, L \in\KK_n$ and $a,b \geq 0$, define the {\em Blaschke combination} 
$a \hspace{-1mm}\cdot\hspace{-1mm} K \, \bs \, b \hspace{-1mm}\cdot\hspace{-1mm} L$ 
to be the unique convex body (up to translation) such that
$$S_{a \cdot K \, \bs \, b \cdot L} = aS_K + bS_L.$$
Although the Blaschke sum $K \bs L$ is identical (up to translation) 
to the Minkowski sum $K+L$ for convex bodies
$K$ and $L$ in $\RR^2$, the two sums are
substantially different for bodies in $\RR^n$ where $n \geq 3$.  
Moreover, for dimension $n \geq 3$, the scalar multiplication 
$a \hspace{-1mm}\cdot\hspace{-1mm} K$ also differs from the usual 
scalar multiplication $aK$ used with Minkowski combinations.
Specifically, $a \hspace{-1mm}\cdot\hspace{-1mm} K = a^{\frac{1}{n-1}}K$, 
since surface area in $\RR^n$ 
is homogeneous of degree $n-1$.

It follows from~(\ref{anyvol}) that, for $K,L,M \in \KK_n$ and $a,b \geq 0$,
\begin{equation}
V_{n-1, 1}(a \hspace{-1mm} \cdot \hspace{-1mm} K \, \bs \, 
b \hspace{-1mm} \cdot \hspace{-1mm} L, M) = aV_{n-1, 1}(K, M) + b V_{n-1, 1}(L, M).
\label{right}
\end{equation}
Note the important difference between~(\ref{left}) and~(\ref{right}) for $n \geq 3$.

It is not difficult to show that every polytope is a Blaschke combination of a finite number of simplices,
while every centrally symmetric polytope is Blaschke combination of a finite number of parallelotopes
(i.e., affine images of cubes) \cite[p. 334]{grunbaum}.  A standard continuity argument (using the Minkowski Existence Theorem and the selection principle for convex bodies \cite[p. 50]{red}) then implies that every convex body
can be approximated (in the Hausdorff topology) by Blaschke combinations of simplices, while
every centrally symmetric convex body
can be approximated by Blaschke combinations of parallelotopes.

A brief and elegant discussion of Blaschke sums, their properties, 
and applications, can be also be found in \cite{Lut1988}.

\section{A counterexample}
\label{counter}

We will exhibit convex bodies $K$ and $L$ in $\RR^n$, such that $V_n(K) > V_n(L)$,
while the orthogonal projection $L_u$ contains a translate of
the corresponding projection $K_u$ for each unit direction $u$.

Note that a suitable disk and Reuleaux triangle provide a well-known counterexample in the 2-dimensional case.  This section provides examples for
bodies of dimension $n \geq 3$.

For $K \in \KK_n$, denote by $r_K$ the {\em inradius} of $K$; that is, the maximum radius taken over all Euclidean balls inside $K$.  Denote by $d_K$ the {\em minimal width} of $K$; that is, the minimum length taken over all orthogonal projections of $K$ onto lines through the origin.  The minimal width is also equal to the minimum distance between any two parallel support planes for $K$.

Let $\Delta$ denote the $n$-dimensional regular simplex having unit edge length.  
The following well-known statistics will be used in the construction that follows:

\begin{center}
\begin{tabular}{rlll}
$\tau_n$ &= \, volume of $\Delta$ &=  \,
$\displaystyle{\frac{ \sqrt{n+1}}{2^{n/2} \, n!}}$&\\[6mm]
$S(\Delta)$ &= \, surface area of $\Delta$ &=  \,
$\displaystyle{\frac{(n+1) \sqrt{n}}{2^{\frac{n-1}{2}} (n-1)!}}$ &= \,
$(n+1) \tau_{n-1}$ \\[6mm]
$r_\Delta$ &= \, inradius of $\Delta$ &=  \, $\displaystyle{\frac{1}{\sqrt{2n(n+1)}}}$
&=  \, $\displaystyle{\frac{n \tau_n}{S(\Delta)}}$\\[8mm]
\end{tabular}
\end{center}
and
\begin{equation}
d_\Delta = \hbox{ minimal width of } \Delta  = 
\left\{
\begin{array}{ll}
\displaystyle{\tfrac{2(n+1)}{\sqrt{n+2}}}\; r_\Delta & \\[6mm]
\displaystyle{2\sqrt{n}}\; r_\Delta & 
\end{array}
\right.
 = \left\{
\begin{array}{ll}
\displaystyle{\sqrt{\tfrac{2(n+1)}{n(n+2)}}} & \hbox{ if } n \hbox{ is even}\\[6mm]
\displaystyle{\sqrt{\tfrac{2}{n+1}}}\;  & \hbox{ if } n \hbox{ is odd}
\end{array}
\right.
\label{dsimp}
\end{equation}
See, for example, \cite[p. 86]{Bonn2}.

To construct and verify the counterexample it will be necessary to compare the minimal width and inradius of a regular simplex with those of its lower-dimensional projections.  
{\em Steinhagen's inequality} asserts that, for $K \in \KK_n$, \\
\begin{equation}
r_K \geq 
\left\{
\begin{array}{ll}
\displaystyle{\tfrac{\sqrt{n+2}}{2n+2}}\; d_K & \hbox{ if } n \hbox{ is even}\\[4mm]
\displaystyle{\tfrac{1}{2\sqrt{n}}}\; d_K & \hbox{ if } n \hbox{ is odd}
\end{array}
\right.
\label{steinhagen}
\end{equation}
A proof of~(\ref{steinhagen}) is given in \cite[p. 86]{Bonn2}.  
If $u$ is a unit vector, then the orthogonal
projection $\Delta_u$ satisfies $d_{\Delta_u} \geq d_\Delta$, where
$d_{\Delta_u}$ is now computed from within the $(n-1)$-dimensional subspace $u^{\perp}$.
Since
$\dim(\Delta_u) = n-1$ has parity opposite that of $n$, it follows 
from~(\ref{steinhagen}) that
$$r_{\Delta_u} \geq 
\left\{
\begin{array}{l}
\displaystyle{\tfrac{\sqrt{n+1}}{2n}}\; d_{\Delta_u} \\[4mm]
\displaystyle{\tfrac{1}{2\sqrt{n-1}}}\; d_{\Delta_u} 
\end{array}
\right.
\geq
\left\{
\begin{array}{ll}
\displaystyle{\tfrac{\sqrt{n+1}}{2n}}\; d_{\Delta} & \hbox{ if } n \hbox{ is odd}\\[4mm]
\displaystyle{\tfrac{1}{2\sqrt{n-1}}}\; d_{\Delta} & \hbox{ if } n \hbox{ is even}
\end{array}
\right.
$$
Combining this with~(\ref{dsimp}) yields
\begin{equation}
r_{\Delta_u} \geq 
\left\{
\begin{array}{ll}
\displaystyle{\tfrac{1}{n\sqrt{2}}} & \hbox{ if } n \hbox{ is odd} \\[4mm]
\displaystyle{\tfrac{\sqrt{n+1}}{\sqrt{2}\sqrt{n(n-1)(n+2)}}} & \hbox{ if } n \hbox{ is even}
\end{array}
\right\}
\; \geq \; \frac{1}{n\sqrt{2}}
\label{evodd}
\end{equation}\\

Let $B_n$ denote the $n$-dimensional 
Euclidean ball centered at the origin and having unit radius.
For $0 \leq \epsilon \leq 1$, denote 
$$K^{\epsilon} = \epsilon \Delta + \left(\tfrac{1-\epsilon}{n\sqrt{2}}\right) B_n.$$

\begin{proposition} For each unit vector $u$ in $\RR^n$,
there exists $v \in u^{\perp}$ such that
$$K^{\epsilon}_u + v \subseteq \Delta_u$$
\label{1}
\end{proposition}
\noindent
In other words, each shadow of the simplex $\Delta$ 
contains a translate of the corresponding shadow of $K^\epsilon$.
\begin{proof} Let $u$ be a unit vector in $\RR^n$.
Since $\frac{1}{n \sqrt{2}} \leq r_{\Delta_u}$, there exists $w \in u^\perp$
such that
$$\tfrac{1}{n \sqrt{2}} B_{n-1} \subseteq \Delta_u - w.$$
Hence,
$$K^{\epsilon}_u  = \epsilon \Delta_u + (1 - \epsilon) \tfrac{1}{n\sqrt{2}} B_{n-1}
\subseteq
\epsilon \Delta_u + (1 - \epsilon)(\Delta_u - w) = 
\Delta_u + (\epsilon-1)w.$$
Setting $v = (1-\epsilon) w$, we have 
$K^{\epsilon}_u + v \subseteq \Delta_u$.
\end{proof}

Next, recall from Steiner's formula~(\ref{steinform}) that if $K$ is a convex body in $\RR^n$ then
\begin{equation}
V_n(\epsilon K + \alpha B_n) = 
\epsilon^n V_n(K) + \epsilon^{n-1} \alpha S(K) + \alpha^2 f(\alpha, \epsilon),
\label{stein-ball}
\end{equation}
where $f(\alpha, \epsilon)$ is a polynomial in $\alpha$ and $\epsilon$ having non-negative coefficients.

\begin{proposition} If $1-\epsilon > 0$ is sufficiently small, then
$V_n(K^\epsilon) > V_n(\Delta)$. 
\label{2}
\end{proposition}

\begin{proof} We need to show that $V_n(K^{\epsilon}) - V_n(\Delta) > 0$.  Applying~(\ref{stein-ball}) 
yields
\begin{eqnarray*}
V_n(K^{\epsilon}) - V_n(\Delta) 
&=& V_n \left( \epsilon \Delta + \tfrac{1-\epsilon}{n \sqrt{2}}B_n \right) - V_n(\Delta)\\
&=& (\epsilon^n - 1) V_n(\Delta) + 
\epsilon^{n-1} \left(\tfrac{1-\epsilon}{n \sqrt{2}} \right) S(\Delta) + 
\left(\tfrac{1-\epsilon}{n \sqrt{2}} \right)^2 f_n(\epsilon) \\
&=& (\epsilon^n - 1) \left( \tfrac{\sqrt{n+1}}{2^{n/2} \, n!}\right) + 
\epsilon^{n-1} \left( \tfrac{1-\epsilon}{n \sqrt{2}} \right) 
\left( \tfrac{(n+1)\sqrt{n}}{2^{\frac{n-1}{2}} (n-1)!} \right) + 
\left( \tfrac{1-\epsilon}{n \sqrt{2}} \right)^2 f_n(\epsilon) \\
&=&(\epsilon^n - 1) \left( \tfrac{\sqrt{n+1}}{2^{n/2} \, n!}\right) + 
\epsilon^{n-1}(1-\epsilon)
\left( \tfrac{(n+1)\sqrt{n}}{2^{n/2} \, n!} \right)+ 
\left(\tfrac{1-\epsilon}{n \sqrt{2}} \right)^2 f_n(\epsilon) \\
\end{eqnarray*}
where $f_n(\epsilon)$ is a polynomial in $\epsilon$.

It follows that $V_n(K^{\epsilon}) - V_n(\Delta) > 0$ if and only if
$$\epsilon^{n-1}(1-\epsilon)
\left( \tfrac{(n+1)\sqrt{n}}{2^{n/2} \, n!} \right)+ 
\left(\tfrac{(1-\epsilon)^2}{2n^2} \right) f_n(\epsilon) 
> (1-\epsilon^n) \left( \tfrac{\sqrt{n+1}}{2^{n/2} \, n!} \right)
$$
if and only if
\begin{equation}
\epsilon^{n-1} \sqrt{n(n+1)} + 
2^{n/2} \, n! \left(\tfrac{1-\epsilon}{2n^2 \sqrt{n+1}} \right) f_n(\epsilon) 
> (1+ \epsilon + \epsilon^2 + \cdots + \epsilon^{n-1}) 
\label{bleep}
\end{equation}
As $\epsilon \rightarrow 1$, the left-hand side of~(\ref{bleep}) approaches $\sqrt{n(n+1)}$,
while the right-hand side approaches $n$, a strictly smaller value for all
positive integers $n$.  It follows that
$V_n(K_\epsilon) > V_n(\Delta)$ for $\epsilon$ sufficiently close to 1.
\end{proof}

Propositions~\ref{1} and~\ref{2} imply that if $0 < \epsilon < 1$ is sufficiently close to $1$,
then every shadow of $\Delta$ contains a translate of the corresponding
shadow of the body $K^\epsilon$, even though $V_n(K^\epsilon) > V_n(\Delta)$.

More precise conditions on admissible values of $\epsilon$ depend on $n$.  
For the case $n=3$ the inequalities used in the proof of Proposition~\ref{2},
along with some additional very crude estimates, imply that $\epsilon = 0.9$ gives a specific
counterexample.  In other words, the 3-dimensional convex bodies:
$$K = \tfrac{9}{10}\Delta_3 + \tfrac{1}{30\sqrt{2}}B_3 
\;\;\;\; \hbox { and } \;\;\;\; \Delta_3$$
have the property that each shadow of $K$ can be covered by a translate of 
the corresponding shadow of the unit regular tetrahedron $\Delta_3$, 
even though $K$ has {\em greater} 
volume\footnote{A more precise calculation yields $V_3(K) \approx 0.122$ and $V_3(\Delta_3) \approx 0.118$.} 
than $\Delta_3$.

At this point one might ask whether suitable conditions on either of the bodies $K$ and $L$ might guarantee that covering shadows implies larger volume.  It is not difficult to show that if $L$ is centrally symmetric, then
$L$ will have greater volume than $K$ when the shadows of $L$ can cover those of $K$.  To see this, suppose that $L=-L$.
If $K_u \subseteq L_u + v$, then 
$-K_u \subseteq -L_u - v = L_u - v$
so that
$$K_u + (-K_u) \subseteq L_u + v + L_u - v = L_u$$
for every direction $u$.  It follows that $K+(-K) \subseteq L+L = 2L$.   
Monotonicity of volume and the Brunn-Minkowski Inequality~(\ref{bmcc}) then imply that
$$V_n(L)^{1/n} \geq V_n \left(\tfrac{1}{2} K + \tfrac{1}{2}(-K) \right)^{1/n} \geq 
\tfrac{1}{2} V_n(K)^{1/n} + \tfrac{1}{2}V_n(-K)^{1/n} = V_n(K)^{1/n},$$
so that $V_n(L) \geq V_n(K)$.

This volume inequality also turns out hold when $L$ is chosen from a much larger family of bodies, 
to be described in Section~\ref{pos}.

\section{A more general counterexample}
\label{gencounter}

The counterexample of Section~\ref{counter} will now be generalized.  
If $\xi$ is a $k$-dimensional subspace
of $\RR^n$, denote by $K_\xi$ the orthogonal projection 
of a set $K \subseteq \RR^n$ to the subspace $\xi$.  For $0 \leq m \leq n$ denote by $V_m(K)$ the
$m$th intrinsic volume of $K$.   The intrinsic volume functional $V_m$ restricts to $m$-dimensional volume on
$m$-dimensional convex sets and is proportional to the mean $m$-volume of the 
$m$-dimensional orthogonal projections of $K$ for more general $K \in \KK_n$.  See, for example,
\cite{Lincee} or \cite[p. 210]{red}.

The following lemma is helpful for extending some low dimensional constructions to higher dimension.
\begin{lemma}  
Suppose that $K$ and $L$ are compact convex sets in $\RR^j \subseteq \RR^n$, where $j \leq n$,
and suppose that
$L_\xi$ can cover $K_\xi$ for all $i$-subspaces $\xi$ of $\RR^j$.  
Then 
$L_\xi$ can cover $K_\xi$ for all $i$-subspaces  $\xi$ of $\RR^{n}$.
\label{embedproj}
\end{lemma}
In other words, if the $i$-dimensional shadows of $L$ can cover those of $K$ in $\RR^j$, then
this covering relation is preserved when $K$ and $L$ are embedded together 
(along with $\RR^j$) in the higher dimensional
space $\RR^n$.  
\begin{proof} Suppose that $K$ and $L$ are compact convex sets in $\RR^j$,
and suppose that $L_\xi$ can cover $K_\xi$ for all $i$-subspaces $\xi$ of $\RR^j$.

Suppose that $\eta$ is an $i$-dimensional subspace of $\RR^{j+1}$.  Then $\dim(\eta^{\perp}) = j-i+1$, 
and $\dim(\eta^{\perp} \cap \RR^j) = j-i$ for generic choices of $\eta$.  Assume $\eta$ is chosen this way.

Let $\xi$ denote the orthogonal complement of $\eta^{\perp} \cap \RR^j$ taken within $\RR^j$.  Since
$\dim(\xi) = i$, there exists $v \in \RR^j$ such that $(K + v)_\xi \subseteq L_\xi$, by the covering assumption
for $K$ and $L$ in $\RR^j$.  This means that, for each $x \in K + v$, there exists $y \in L$ such that
$x-y$ is orthogonal to $\xi$.  It follows from the construction of $\xi$ that $x-y \in \eta^{\perp}$.

Hence, for all $x \in K + v$, there exists $y \in L$ such that $x-y \in \eta^\perp$.  This implies that
$K_\eta + v_\eta \subseteq L_\eta$.

We have shown that $L_\eta$ can cover $K_\eta$ for all $i$-subspaces $\eta$ of $\RR^{j+1}$
such that $\dim(\eta^{\perp} \cap \RR^j) = j-i$.  Since this is a dense family of $i$-subspaces,
the lemma follows more generally for all $i$-subspaces of $\RR^{j+1}$.  By a suitable iteration of this argument, the lemma
then follows for $i$-subspaces of $\RR^{n}$, for any $n > j$.
\end{proof}

We can now generalize the counterexample of Section~\ref{counter}.

\begin{theorem} Let $n \geq 3$ and $1 \leq k < n$.
There exist convex bodies $K, L \in \KK_n$ such that $L_\xi$ can cover $K_\xi$ for all 
$k$-dimensional subspaces $\xi$, while 
$V_m(K) > V_m(L)$ for all $m > k$.
\label{counters}
\end{theorem}

Note that this theorem is already well-known for the case $k = 1$.  In that particular case,
the covering condition merely asserts that the width of $K$ in any direction is smaller
than or equal to the corresponding width of $L$.
The novel aspect of this result addresses the cases in which $2 \leq k < n$.

\begin{proof}  If $k = 1$, then let $K$ be an $n$-simplex, and let $M = \tfrac{1}{2}K + (-\tfrac{1}{2}K)$, the
difference body of $K$.  It then follows from the Minkowski mixed volume inequality~(\ref{mmv}) and
the classical mean projection formulas for intrinsic volumes \cite{Gard2006,Lincee,red} that 
$V_m(M) > V_m(K)$ for $m \geq 2$, while $K$ and $M$ have identical width in every direction.

Now suppose that $k \geq 2$.
Let $\hat{K}$ and $\hat{L}$ be chosen in $\KK_{k+1}$ so that
$\hat{L}_\xi$ can cover $\hat{K}_\xi$ for all $k$-subspaces $\xi$ of $\RR^{k+1}$, while
$V_{k+1}(\hat{K}) > V_{k+1}(\hat{L})$.  (One could follow the explicit 
construction given in Section~\ref{counter}, for example.) 

If $n = k+1$, we are done.  If $n > k+1$, 
embed $\hat{K}$ and $\hat{L}$ in $\RR^n$ via the usual coordinate embedding of $\RR^{k+1}$ in $\RR^n$.  Then 
$\hat{L}_\xi$ can cover $\hat{K}_\xi$ for all $k$-subspaces $\xi$ of $\RR^{n}$, by Lemma~\ref{embedproj}.

Let $C$ denote the unit cube in $\RR^{n-k-1}$ with edges parallel to the standard axes in the orthogonal
complement to $\RR^{k+1}$ in $\RR^n$. 
Let $K = \hat{K} + \epsilon C$ and $L = \hat{L} + \epsilon C$.  Then
$L_\xi$ can cover $K_\xi$ for all 
$k$-dimensional subspaces $\xi$ once again, since
$K_\xi = \hat{K}_\xi + \epsilon C_\xi$, and similarly for $L$.

Moreover, if $m \geq k+1$ then
$$
V_m(K) \;=\; V_m(\hat{K} + \epsilon C) \;  = \sum_{i+j = m} V_i(\hat{K})V_j(C) \epsilon^j\\
$$
by the Cartesian product formula for intrinsic volumes \cite[p. 130]{Lincee}.  Hence,
\begin{align*}
V_m(K) & = \; \sum_{i=0}^{k+1} V_i(\hat{K})V_{m-i}(C) \epsilon^{m-i}\\
& = \; \sum_{i=0}^{k+1} \binom{n-k-1}{m-i} V_i(\hat{K}) \epsilon^{m-i}\\
& = \; \epsilon^{m-k-1} \binom{n-k-1}{m-k-1} V_{k+1}(\hat{K}) + f_K(\epsilon)
\end{align*}
where $f_K(\epsilon)$ is a polynomial in $\epsilon$ composed of monomials having 
degree greater than $m-k-1$.  A similar formula holds
for $V_m(L)$.  Therefore,
$$V_m(K) - V_m(L) = \epsilon^{m-k-1} \binom{n-k-1}{m-k-1} 
\left( V_{k+1}(\hat{K})-V_{k+1}(\hat{L}) \right)
+ (f_K(\epsilon) - f_L(\epsilon)),$$
where $f_K(\epsilon) - f_L(\epsilon)$ 
is a polynomial in $\epsilon$ composed of monomials having 
degree greater than $m-k-1$.
Since the lowest degree coefficient of
the polynomial formula for $V_m(K) - V_m(L)$ is positive, we have $V_m(K) - V_m(L) > 0$ when $\epsilon > 0$
is sufficiently small.
\end{proof}

\section{Cylinders and shadow covering}
\label{sec-cover}

Let $K \in \KK_n$ and suppose that $P \in \KK_n$ is a polytope.  
A {\em facet} of $P$ is a face (support set) of $P$ having dimension $n-1$.
We say that $P$ {\em circumscribes} $K$ if $K \subseteq P$  and $K$ also meets every facet of $P$.
\begin{lemma}[Circumscribing Lemma]  Let $K, P \in \KK_n$, where $P$ is a polytope.  
If $P$ circumscribes $K$ then
\begin{equation}
V_{n-1,1}(P, K) = V_n(P).
\label{circumcond}
\end{equation}
If we are given that $K \subseteq P$, then~(\ref{circumcond}) holds if and only if $P$ circumscribes $K$.
\label{clem}
\end{lemma}

\begin{proof}  
If $K \subseteq P$ and if
$K$ meets every facet of $P$, then $h_K(u) = h_P(u)$ whenever the direction $u$ is normal to a facet of $P$.
Since $P$ is a polytope, the mixed volume formula~(\ref{polyvol}) yields
$$V_{n-1,1}(P, K) = \frac{1}{n} \sum_{u \perp \partial P} h_K(u) V_{n-1}(P^u) = 
\frac{1}{n} \sum_{u \perp \partial P} h_P(u) V_{n-1}(P^u) = V(P).
$$

Conversely,
if we are given that $K \subseteq P$, then $h_K(u) \leq h_P(u)$, with equality for all facet normals $u$ if and only if
$P$ circumscribes $K$, so that~(\ref{circumcond}) holds if and only if $P$ circumscribes $K$. 
\end{proof}

The case in which $P$ is a simplex is especially important, 
because of the following theorem of Lutwak \cite{Lut-contain} (see also \cite{Lincee}), 
itself a consequence of Helly's theorem.
\begin{theorem}[Lutwak's Containment Theorem]  Let $K, L \in \KK_n$.  Suppose that, for every simplex $\Delta$ such 
that $L \subseteq \Delta$, there is a vector $v_\Delta \in \RR^n$ such that $K + v_\Delta \subseteq \Delta$.  
Then there is a vector
$v \in \RR^n$ such that $K +v \subseteq L$.
\label{lcon}
\end{theorem}
Lutwak's theorem combines with the Circumscribing Lemma to yield the following useful corollary 
(also from \cite{Lut-contain}).
\begin{corollary} Let $K, L \in \KK_n$.  The inequality
$$V_{n-1,1}(\Delta, K) \leq V_{n-1,1}(\Delta, L)$$
holds for all simplices $\Delta$, if and only if there exists $v \in \RR^n$ such that $K+v \subseteq L$.
\label{lcor}
\end{corollary}

Suppose that $\lambda_1 \geq \lambda_2 \geq \ldots \geq \lambda_m > 0$ are positive integers
such that $\lambda_1 + \lambda_2 + \ldots + \lambda_m = n$.  
Denote $\lambda = (\lambda_1, \lambda_2, \ldots, \lambda_m)$.  The vector $\lambda$ is sometimes called a {\em partition} of the positive integer $n$.  Using this notation, the size of the largest part of any partition 
$\lambda$ is given by the first entry $\lambda_1$.

A convex body $K \in \KK_n$ will be called {\em $\lambda$-decomposable} if there exists affine 
subspaces $\xi_i$ of $\RR^n$ such that $\dim \xi_i = \lambda_i$ and 
$\RR^n = \xi_1 \oplus \cdots \oplus \xi_m$ and if there exists compact convex sets
$K_i \subseteq \xi_i$ such that $K = K_1 + \cdots + K_m$.  In this case we will write
$K = K_1 \oplus \cdots \oplus K_m$

The body $K$ will be called
{\em $\lambda$-ortho-decomposable} if $\xi_i \perp \xi_j$ for each $i \neq j$.
For example, a cylinder is an $(n-1,1)$-decomposable body.  
A $(1, 1, \ldots, 1)$-decomposable body is a parallelotope, while
a $(1, 1, \ldots, 1)$-ortho-decomposable body is an orthogonal box.  

For $k \in \{ 1, \ldots, n-1 \}$, denote by $G(n,k)$ the collection of all
$k$-dimensional linear subspaces $\xi$ of $\RR^n$, sometimes called the
$(n,k)$-Grassmannian.  For $\xi \in G(n,k)$ and $K \in \KK_n$, denote by
$K_{\xi}$ the orthogonal projection of the body $K$ onto the subspace $\xi$.

Lutwak's Containment Theorem~\ref{lcon} and its Corollary~\ref{lcor} lead to the following useful condition for determining when the shadows of one body can cover those of another.

\begin{theorem}[First Shadow Containment Theorem]
Let $K, L \in \KK_n$, and suppose that $1 \leq k \leq n-1$.
The orthogonal projections $L_\xi$ of $L$ can cover the corresponding projections
$K_\xi$ of $K$ for all $\xi \in G(n,k)$  if and only if
$$V_{n-1,1}(C, K) \leq V_{n-1,1}(C, L)$$
for all $\lambda$-ortho-decomposable $C \in \KK_n$ such that $\lambda_1 \leq k$.
\label{orthlambda}
\end{theorem} 

\begin{proof}
Suppose that $C$ is a $\lambda$-ortho-decomposable polytope, with orthogonal
decomposition $C = a_1 C_1 \oplus \cdots \oplus a_m C_m$, where each
$C_i$ has affine hull parallel to a subspace $\xi_i$ of dimension $\lambda_i$
and $a_1, \ldots, a_m > 0$.
Note that 
$$V_n(C) \, = \, V_{\lambda_1}(a_1 C_1) \cdots V_{\lambda_m}(a_m C_m)
\, = \,  a_1^{\lambda_1} \cdots a_m^{\lambda_m} V_{\lambda_1}(C_1)
\cdots V_{\lambda_m}(C_m),$$ 
since the decomposition is orthogonal.

It follows from~(\ref{polyvol}) that
$$V_{n-1,1}(C, K) 
= \frac{1}{n} \sum_{u \perp \, \partial C} h_K(u) V_{n-1}(C^u)$$
where the sum is taken over all unit directions $u \in \RR^n$ normal to facets of $C$.
The product structure of $C$ then implies that
$$
V_{n-1,1}(C, K) 
= \frac{1}{n} \sum_{i=1}^m \sum_{u \perp \, \partial C_i} h_K(u) V_{\lambda_1}(a_1 C_1) 
\cdots V_{\lambda_i-1}(a_{i} C_i^u) \cdots V_{\lambda_m}(a_m C_m)\\[1mm]
$$
where, for each $i$, the inner sum is taken over all unit directions $u \in \xi_i$ normal to facets of $C_i$.
Hence,
\begin{align*}
V_{n-1,1}(C, K) 
&= \frac{1}{n} \sum_{i=1}^m \frac{V_n(C)}{V_{\lambda_i}(a_i C_i)} 
\sum_{u \perp \, \partial C_i} h_K(u) V_{\lambda_1 -1}(C_i^u) a_i^{\lambda_i -1}\\[1mm]
&= \frac{1}{n} \sum_{i=1}^m \frac{iV_n(C)}{a_i V_{\lambda_i}(C_i)} \, 
V_{\lambda_i-1,1}(C_i,K_{\xi_i})\\[1mm]
\end{align*}
for all $a_1, \ldots, a_m > 0$.

If $L_\xi$ can cover $K_\xi$ for all $\xi \in G(n,k)$, then 
$L_\eta$ also can cover $K_\eta$ for all lower dimensional 
subspaces $\eta \in G(n,j)$, where $1 \leq j \leq k$.
In particular $L_{\xi_i}$ can cover $K_{\xi_i}$ for each $i$, 
since $\dim \xi_i = \lambda_i \leq \lambda_1 = k$.  It follows that each 
$V_{i-1,1}(C_i,K_{\xi_i}) \leq V_{i-1,1}(C_i,L_{\xi_i})$
by the monotonicity and translation invariance of mixed volumes.  Therefore, 
$V_{n-1,1}(C, K) \leq V_{n-1,1}(C, L)$ for all $\lambda$-ortho-decomposable polytopes $C$.
The inequality then holds for arbitrary $\lambda$-ortho-decomposable $C$ by continuity of mixed volumes.

Conversely, if $V_{n-1,1}(C, K) \leq V_{n-1,1}(C, L)$ 
for all $\lambda$-ortho-decomposable $C \in \KK_n$ such that $\lambda_1 \leq k$,
then 
$$
\frac{1}{n} \sum_{i=1}^m \frac{iV_n(C)}{a_i V_{\lambda_i}(C_i)} V_{\lambda_i-1,1}(C_i,K_{\xi_i}) 
\leq \frac{1}{n} \sum_{i=1}^m \frac{iV_n(C)}{a_i V_{\lambda_i}(C_i)} V_{{\lambda_i}-1,1}(C_i,L_{\xi_i})
$$
for all such $C$ and all $a_1, \ldots, a_m > 0$.

In particular, $V_{\lambda_i-1, 1}(\Delta, K_{\xi_i}) \leq V_{\lambda_i-1, 1}(\Delta, L_{\xi_i})$
for every $\lambda_i$-simplex $\Delta$ 
in every $\lambda_i$-dimensional subspace $\xi_i$ of $\RR^n$
so that $L_{\xi_i}$ can cover $K_{\xi_i}$ by Corollary~\ref{lcor}.
\end{proof}

\section{Cylinder bodies and shadow covering}
\label{sec-cyl}

So far we have restricted attention to {\em orthogonal} cylinders and decomposable sets.  
However, the previous results generalize easily 
to arbitrary (possibly oblique) cylinders and decompositions.

For $S \subseteq \RR^n$ and a nonzero vector $u$, let $\LL_S(u)$ denote the set of straight lines in $\RR^n$
parallel to $u$ and meeting the set $S$.
\begin{proposition} Let $K, L \in \KK_n$.  Let $\psi: \RR^n \rightarrow \RR^n$ be a
non-singular linear transformation.  Then $L_u$ contains a translate of $K_u$ for all unit directions $u$
if and only if $(\psi L)_u$ contains a translate of $(\psi K)_u$ for all $u$.
\label{affcover}
\end{proposition}

\begin{proof} The projection $L_u$ contains a translate of $K_u$ for each unit vector $u$ if and only if,
for each $u$, there exists $v_u$ such that
\begin{equation}
\LL_{K+v_u} (u) \subseteq \LL_L(u).
\label{hum}
\end{equation}
But $\LL_{K+v_u}(u) = \LL_{K}(u) + v_u$ 
and $\psi \LL_{K}(u) = \LL_{\psi K}(\psi u)$.
It follows that~(\ref{hum}) holds if and only if
$\LL_{K}(u) + v_u \subseteq \LL_L(u)$,
if and only if
$$\LL_{\psi K}( \psi u) + \psi v_u \subseteq \LL_{\psi L}(\psi u) \;\;\; \hbox{ for all unit } u,$$
Set 
$$\tilde{u} = \frac{\psi u}{|\psi u|} \;\;\; \hbox{ and } \;\;\; \tilde{v} = \psi v_u.$$
The relation~(\ref{hum}) now holds if and only if,
for all $\tilde{u}$, there exists $\tilde{v}$ such that
$$\LL_{\psi K}(\tilde{u}) + \tilde{v} \subseteq \LL_{\psi L}(\tilde{u}),$$
which holds if and only if $(\psi L)_{\tilde{u}}$ 
contains a translate of $(\psi K)_{\tilde{u}}$ for all $\tilde{u}$.
\end{proof}

We are now in a position to define a much larger family of objects
that serve to generalize the Shadow Containment Theorem~\ref{orthlambda}.
 
\begin{definition}
For each $k \in \{1, \ldots, n\}$ denote by 
$\cy_{n,k}$ set of all bodies $K \in \KK_n$ 
that can be approximated (in the usual Hausdorff topology) by Blaschke combinations of 
$\lambda$-decomposable sets for any $\lambda$ such that $\lambda_1 \leq k$.
Elements of $\cy_{n,k}$ will be called the $k$-cylinder bodies of $\RR^n$.
\end{definition}

Recall that any centrally symmetric polytope is a Blaschke sum of parallelotopes.
It follows that $\cy_{n,1}$ is precisely the set of all centrally symmetric convex bodies in $\RR^n$.  
For $n \geq 3$ and $k \geq 2$, the cylinder bodies $\cy_{n,k}$ 
are a larger family of objects.  For example, a triangular cylinder in $\RR^3$ lies in
$\cy_{3,2}$, but not in $\cy_{3,1}$, since it is not centrally symmetric.  
Note also that $\cy_{n,k}$ is closed under affine transformations.

The definition of $\cy_{n,k}$ depends on the ambient dimension $n$ as well as the value $k$,
because the notion of Blaschke sum $\bs$ depends on $n$.  For example, while 
Minkowski sum satisfies the projection identity $(K+L)_\xi = K_\xi + L_\xi$ for subspaces $\xi \subseteq \RR^n$,
the analogous statement need not hold for Blaschke summation.

Note also that $\cy_{n,n} = \KK_n$ by definition.  Moreover, it follows from the 
definition that $\cy_{n,i} \subseteq \cy_{n,j}$ whenever $i \leq j$.  
It will be shown in Section~\ref{pos}  
that $\cy_{n,i}$ is a {\em proper} subset of $\cy_{n,j}$ when $i < j$.
In particular, it will be seen that
full-dimensional simplices
are {\em not} $k$-cylinder bodies of $\RR^n$ for any $k < n$.
A necessary condition for being a $k$-cylinder body will be
described in Section~\ref{sec-ineq}. 

The significance of each collection $\cy_{n,k}$ is described in part by the following theorem.
\begin{theorem}[Second Shadow Containment Theorem]  Let $K, L \in \KK_n$
and let $1 \leq k \leq n$. 
The following are equivalent:
\begin{enumerate}
\item[\bf (i)] The orthogonal projections $L_\xi$ of $L$ can cover the corresponding projections
$K_\xi$ of $K$ for all subspaces $\xi \in G(n,k)$.\\[-2mm]

\item[\bf (ii)] The affine projections $\pi L$ of $L$ can cover the corresponding projections
$\pi K$ of $K$ for all affine projections $\pi$ of rank $k$.\\[-2mm]

\item[\bf (iii)] $V_{n-1,1}(C, K) \leq V_{n-1,1}(C, L)$ for all 
{\em $\lambda$-ortho-decomposable} sets $C$ such that $\lambda_1 \leq k$.\\[-2mm]

\item[\bf (iv)] $V_{n-1,1}(C, K) \leq V_{n-1,1}(C, L)$ for all 
{\em $\lambda$-decomposable} sets $C$ such that $\lambda_1 \leq k$.\\[-2mm]

\item[\bf (v)] $V_{n-1,1}(Q, K) \leq V_{n-1,1}(Q, L)$ for all $k$-cylinder bodies $Q \in \cy_{n,k}$.
\end{enumerate}
\label{omni-k}
\end{theorem}

\begin{proof}  The equivalence of {\bf (i)} and {\bf (ii)} follows from Proposition~\ref{affcover}. 
The equivalence of {\bf (i)} and {\bf (iii)} follows from Theorem~\ref{orthlambda}. 

To show that {\bf (iii)} implies {\bf (iv)}, suppose that {\bf (iii)} holds for the pair $K,L$.  
It follows from {\bf (i)} and Proposition~\ref{affcover} that {\bf (i)} also holds for the 
pair of bodies $\psi^{-1} K, \psi^{-1} L$, for any non-singular affine transformation $\psi$.  
Therefore {\bf (iii)} also 
holds for the 
pair of bodies $\psi^{-1} K, \psi^{-1} L$; that is,
$$V_{n-1,1}(C , \psi^{-1} K) \leq V_{n-1,1}(C , \psi^{-1} L).$$
for all $\lambda$-ortho-decomposable sets $C$ such that $\lambda_1 \leq k$.
Let us suppose that $\psi$ has unit determinant.
Then $V_{n-1,1}(\psi C , K) = V_{n-1,1}(C , \psi^{-1} K)$, and similarly for $L$, 
by the affine invariance of (mixed) volumes, so that
$$V_{n-1,1}(\psi C , K) \leq V_{n-1,1}( \psi C ,  L).$$
for all $\lambda$-ortho-decomposable sets $C$ such that $\lambda_1 \leq k$.
If $C'$ is a $\lambda$-decomposable set, then
$C' = \psi C$ for some $\lambda$-ortho-decomposable set $C$ 
and some affine transformation $\psi$ of unit determinant.  
{\bf (iv)} now follows.

{\bf (iv)} implies {\bf (v)} by the Blaschke-linearity of the functional 
$V_{n-1, 1}(\cdot, \cdot)$ in its {\em first} parameter and the continuity of $V_{n-1, 1}$.

Finally, {\bf (v)} implies {\bf (iv)}, and {\bf (iv)} implies {\bf (iii)}, 
in both cases {\em a fortiori}.  
\end{proof}

\section{A positive answer for covering cylinder bodies}
\label{pos}

In Section~\ref{gencounter} we described examples of convex bodies $K$ and $L$ such that 
the orthogonal projections $L_\xi$ of $L$ covered the corresponding projections
$K_\xi$ of $K$ for all $\xi \in G(n,k)$, even though $V_n(L) < V_n(K)$.
The next theorem shows that this volume anomaly can be avoided if $L \in \cy_{n,k}$.

\begin{theorem}  Let $K, L \in \KK_n$ and let $1 \leq k \leq n-1$.
Suppose that the orthogonal projections 
$L_\xi$ of $L$ can cover the corresponding projections
$K_\xi$ of $K$ for all $\xi \in G(n,k)$.  
If $L \in \cy_{n,k}$, then $V_n(K) \leq V_n(L)$. 

If, in addition, the set $L$ has non-empty interior, 
then $V_n(K) = V_n(L)$ if and only if $K$ and $L$ are translates.
\label{cylvol}
\end{theorem}

\begin{proof}  If the orthogonal projections 
$L_\xi$ of $L$ can cover the corresponding projections
$K_\xi$ of $K$ for all $\xi \in G(n,k)$, then
$$V_{n-1,1}(Q,  K) \leq V_{n-1,1}(Q,  L)$$
for all $Q \in \cy_{n,k}$, by Theorem~\ref{omni-k}.  
If $L \in \cy_{n,k}$ as well, then
$$V_{n-1,1}(L, K) \leq V_{n-1,1}(L,  L) = V_n(L)$$
Meanwhile, the Minkowski mixed volume inequality~(\ref{mmv}) asserts that
$$V_n(L)^{(n-1)/n}  V_n(K)^{1/n} \leq V_{n-1,1}(L,  K),$$ 
Hence $V_n(K) \leq V_n(L)$.  If equality holds and $V_n(L) > 0$, then $K$ and $L$ are 
homothetic bodies of the same volume by the equality conditions of~(\ref{mmv}), 
so that $K$ and $L$ must be translates.
\end{proof}

The simplicial counterexamples of Section~\ref{counter} along with Theorem~\ref{cylvol}
yield the following immediate corollary.
\begin{corollary} An $n$-dimensional simplex is never an element of $\cy_{n,n-1}$.
\end{corollary}
In particular, the collection of $(n-1)$-cylinder bodies $\cy_{n,n-1}$
forms a {\em proper} subset of $\cy_{n,n} = \KK_n$.

More generally we have the following.
\begin{corollary} For $1 \leq i < j \leq n$ the set $\cy_{n,i}$ is a proper subset of $\cy_{n,j}$.
\end{corollary}

\begin{proof}  It follows directly from the definition of $\cy_{n,i}$ that
$\cy_{n,i} \subseteq \cy_{n,j}$ when $i < j$.  It remains to show that
$\cy_{n,i} \neq \cy_{n,j}$ when $i < j$.

To see this, observe that 
the set $L$ constructed in the proof of Theorem~\ref{counters} satisfies $L \in \cy_{n,k+1}$,
because $L$ is $\lambda$-decomposable for $\lambda = (k+1, 1, \ldots, 1)$.  Let $K$
also be chosen as in the proof of Theorem~\ref{counters}.  Recall that the $k$-shadows $L_{\xi}$ 
of $L$ can cover those of $K$ for every $\xi \in G(n,k)$.  Since
$V_n(L) < V_n(K)$, it follows from Theorem~\ref{cylvol} that $L \notin \cy_{n,k}$.
Hence $\cy_{n,k} \neq \cy_{n,k+1}$.
\end{proof}

In other words, the collections $\cy_{n,k}$ form a {\em strictly} increasing chain
$$\cy_{n,1} \subset \cy_{n,2} \subset \cdots \subset \cy_{n,n-1} \subset \cy_{n,n} = \KK_n$$
where the elements of $\cy_{n,1}$ are precisely the centrally symmetric sets in $\KK_n$.

\section{A geometric inequality for cylinder bodies}
\label{sec-ineq}

For positive integers $n \geq 2$ denote
$$\sigma_n = \left\{
\begin{array}{ll}
\displaystyle{\tfrac{\sqrt{n+2}}{2n+2}} & \hbox{ if } n \hbox{ is even}\\[4mm]
\displaystyle{\tfrac{1}{2\sqrt{n}}} & \hbox{ if } n \hbox{ is odd}
\end{array}
\right.$$

Recall that we denote the surface area of a convex body $K$ by $S(K)$
and the minimal width of $K$ by $d_K$.

\begin{theorem}[Cylinder body inequality]
Let $K \in \KK_n$.  If $K \in \cy_{n,i}$, then 
\begin{equation}
\sigma_{i} d_K S(K) \leq n V_n(K).
\label{cineq}
\end{equation}
\label{cyl-ineq}
\end{theorem}

\begin{proof} If $V_n(K) = 0$ then $d_K = 0$ as well, so that
both sides of~(\ref{cineq}) are zero.

Suppose that $V_n(K) > 0$.  By Steinhagen's inequality~(\ref{steinhagen}) and the fact that $\dim K = n$, 
$$r_{K_\xi} \geq \sigma_{i} d_{K_\xi} \geq \sigma_{i} d_K,$$
for each subspace $\xi \in G(n,i)$,
where $d_{K_\xi}$ is computed from within the subspace $\xi$.

For $0 \leq \epsilon \leq 1$, denote 
$K^\epsilon = \epsilon K + (1-\epsilon) \sigma_{i} d_K B_n$,
where $B_n$ is an $n$-dimensional unit Euclidean ball.
Since
$$\sigma_{i} d_K B_n \subseteq r_{K_\xi} B_n \subseteq K_\xi \;\;\; \hbox{ up to translation, }$$
we have
$$K^\epsilon_\xi \subseteq \epsilon K_\xi + (1-\epsilon)K_\xi = K_\xi \;\;\; \hbox{ up to translation, }$$
for each subspace $\xi \in G(n,i)$.  If $K$ is an $i$-cylinder body, then
$V_n(K^\epsilon) \leq V_n(K)$, by Theorem~\ref{cylvol}.  Moreover, 
Steiner's formula~(\ref{stein-ball}) implies that 
$$V_n(K^\epsilon) = \epsilon^n V_n(K) + \epsilon^{n-1} (1-\epsilon) \sigma_{i} d_K S(K) 
+ (1-\epsilon)^2 f(\epsilon),$$
where $f(\epsilon)$ is a polynomial in $\epsilon$.
Since $V_n(K^\epsilon) - V_n(K) \leq 0$ for $0 \leq \epsilon < 1$, we have
$$(\epsilon^n -1)V_n(K) + \epsilon^{n-1} (1-\epsilon) \sigma_{i} d_K S(K) 
+ (1-\epsilon)^2 f(\epsilon) \leq 0$$
so that
$$\epsilon^{n-1} \sigma_{i} d_K S(K) 
+ (1-\epsilon) f(\epsilon) \, \leq \, (1+ \epsilon + \cdots + \epsilon^{n-1})V_n(K)$$
for all $0 \leq \epsilon < 1$.  As $\epsilon \rightarrow 1$ this implies
that
$\sigma_{i} d_K S(K) \leq n V_n(K)$.
\end{proof}

\section{A volume ratio bound}
\label{sec-ineq2}

In Section 2 we described $n$-dimensional convex bodies $K$ and $L$ such that the orthogonal projection 
$K_u$ can be translated inside $L_u$ for every direction $u$, while $V_n(K) > V_n(L).$
In such instances, one could ask instead for an upper bound on the volume ratio $\frac{V_n(K)}{V_n(L)}$.
An application of Theorem~\ref{cylvol} yields the following crude estimate.

\begin{theorem}  Let $K, L \in \KK_n$, and
suppose that the orthogonal $(n-1)$-dimensional projections 
$L_u$ of $L$ can cover the corresponding projections
$K_u$ of $K$ for all directions $u$.  
Then $V_n(K) \leq nV_n(L)$. 
\label{crude}
\end{theorem}

Recall that the {\em diameter} $D_K$ of a convex body $K$ is the maximum distance between any two points of the body $K$, and is also equal to the maximum width, that is, the maximum distance between any two parallel supporting hyperplanes of $K$.

\begin{proof} Suppose that the diameter $D_L$ of $L$ is realized in the unit direction $v$.  A standard
Steiner symmetrization (or, alternatively, shaking) argument implies that
$$V_n(L) \geq \tfrac{1}{n} \, D_L V_{n-1}(L_v).$$

Let $\bar{v}$ denote the unit line segment having endpoints at the origin $o$ and at $v$,
and let $C$ be the orthogonal cylinder in $\RR^n$ given by $C = L_v \oplus D_L \bar{v}$.  
After a suitable translation, we may assume that $L \subseteq C$.  From the original covering assumption for $L$ it then follows that each projection $K_u$ can be translated inside the corresponding projection $C_u$ of the cylinder $C$.
By Theorem~\ref{cylvol}, it then follows that
$$V_n(K) \leq V_n(C) = D_L V_{n-1}(L_v) \leq n V_n(L).$$
\end{proof}

\section{Some open questions}

The results of the previous sections motivate 
several open questions about convex bodies and projections.

\begin{enumerate}
\item[I.]  
Let $K, L \in \KK_n$ such that $V_n(L) > 0$, and let $1 \leq k \leq n-1$.
Suppose that the orthogonal projections 
$L_\xi$ of $L$ can cover the corresponding projections
$K_\xi$ of $K$ for all $\xi \in G(n,k)$.  \\

\noindent
What is the best upper bound for the ratio $\frac{V_n(K)}{V_n(L)}$?\\

\item[II.] Given a partition $\lambda$ of a positive integer $n$, define
$\cy_{\lambda}$ to be the collection of all convex bodies that can
be approximated by Blaschke sums of $\mu$-decomposable convex bodies,
taken over all partitions $\mu$ that refine the partition $\lambda$. \\

\noindent
If $\lambda$ and $\sigma$ are incomparable partitions of $n$ 
(with respect to partition refinement),
how are $\cy_{\lambda}$ and $\cy_{\sigma}$ related?  Can we describe their
relative geometric significance in the context of projections?\\

\item[III.]  
Zonoids can be thought of as the image of the projection
body operator on convex sets or of the cosine transform on 
support functions, and intersection
bodies are constructed by taking the Radon transform
of the radial function of a convex (or star-shaped) set
\cite{Gard2006,Lut1988,red}. 
Is there an analogous integral geometric description for
the families $\cy_{\lambda}$ and $\cy_{n,k}$?  \\

\item[IV.]  What simple tests, conditions, or inequalities determine
whether or not a convex body $K$ is an element of 
some $\cy_{\lambda}$ or $\cy_{n,k}$?\\

\item[V.]  Let $K, L \in \KK_n$ such that $V_n(L) > 0$, and let $1 \leq k \leq n-1$.
Suppose that the orthogonal projections 
$L_\xi$ of $L$ can cover the corresponding projections
$K_\xi$ of $K$ for all $\xi \in G(n,k)$.  \\

\noindent
Under what simple (easy to state, easy to verify) additional conditions
does it follow that $K$ can be translated inside $L$?
\end{enumerate}

\providecommand{\bysame}{\leavevmode\hbox to3em{\hrulefill}\thinspace}
\providecommand{\MR}{\relax\ifhmode\unskip\space\fi MR }
\providecommand{\MRhref}[2]{%
  \href{http://www.ams.org/mathscinet-getitem?mr=#1}{#2}
}
\providecommand{\href}[2]{#2}


\begin{thebibliography}{10}

\bibitem{Bonn2}
T.~Bonnesen and W.~Fenchel, \emph{Theory of {C}onvex {B}odies}, BCS Associates,
  Moscow, Idaho, 1987.

\bibitem{bus-pet}
H.~Busemann and C.~M. Petty, \emph{Problems on convex bodies}, Math. Scand.
  \textbf{4} (1956), 88--94.

\bibitem{Gard-Busemann}
R.~J.~Gardner, \emph{A positive answer to the {B}usemann-{P}etty problem in three
  dimensions}, Ann. Math. (2) \textbf{140} (1994), 435--447.

\bibitem{Gard-BM}
\bysame, \emph{The {B}runn-{M}inkowski inequality}, Bull. Amer. Math. Soc.
  \textbf{39} (2002), 355--405.

\bibitem{Gard2006}
\bysame, \emph{Geometric {T}omography (2nd {E}d.)}, Cambridge University Press,
  New York, 2006.

\bibitem{Gard-Kold}
R.~J.~Gardner, A.~Koldobsky, and T.~Schlumprecht, \emph{An analytic solution to
  the {B}usemann-{P}etty problem on sections of convex bodies}, Ann. Math. (2)
  \textbf{149} (1999), 691--703.

\bibitem{grunbaum}
B.~Gr\"{u}nbaum, \emph{Convex {P}olytopes (2nd {E}d.)}, Springer Verlag, New
  York, 2003.

\bibitem{Lincee}
D.~Klain and G.-C. Rota, \emph{Introduction to {G}eometric {P}robability},
  Cambridge University Press, New York, 1997.

\bibitem{Lut1988}
E.~Lutwak, \emph{Intersection bodies and dual mixed volumes}, Adv. Math.
  \textbf{71} (1988), 232--261.

\bibitem{Lut-contain}
\bysame, \emph{Containment and circumscribing simplices}, Discrete Comput.
  Geom. \textbf{19} (1998), 229--235.

\bibitem{Petty-shep}
C.~M. Petty, \emph{Projection bodies}, Proceedings, Coll. Convexity,
  Copenhagen, 1965, vol. K\o benhavns Univ. Mat. Inst., 1967, pp.~234--241.

\bibitem{rogers}
C.~A. Rogers, \emph{Sections and projections of convex bodies}, Portugal. Math.
  \textbf{24} (1965), 99--103.

\bibitem{Schneider-shep}
R.~Schneider, \emph{Zur einem {P}roblem von {S}hephard \"{u}ber die
  {P}rojektionen konvexer {K}\"{o}rper.}, Math Z. \textbf{101} (1967), 71--82.

\bibitem{red}
\bysame, \emph{Convex {B}odies: {T}he {B}runn-{M}inkowski {T}heory}, Cambridge
  University Press, New York, 1993.

\bibitem{Shep}
G.~C. Shephard, \emph{Shadow systems of convex bodies}, Israel J. Math.
  \textbf{2} (1964), 229--236.

\bibitem{Webster}
R.~Webster, \emph{Convexity}, Oxford University Press, New York, 1994.

\bibitem{Zhang-Busemann-old}
G.~Zhang, \emph{Intersection bodies and the {B}usemann-{P}etty inequalities in
  {{\bf R}}$^4$}, Ann. Math. (2) \textbf{140} (1994), no.~2, 331--346.

\bibitem{Zhang-Busemann}
\bysame, \emph{A positive solution to the {B}usemann-{P}etty problem in {{\bf
  R}}$^4$}, Ann. Math. (2) \textbf{149} (1999), 535--543.

\end{thebibliography}
\end{document}